\documentclass{amsart}

\usepackage[foot]{amsaddr}
\usepackage{lmodern}
\usepackage[T1]{fontenc}
\usepackage[utf8]{inputenc}

\usepackage{overpic}
\usepackage{todonotes}

\usepackage{amsmath,amssymb,amsfonts,amsthm}
\usepackage[nobysame]{amsrefs}
\usepackage{mathtools}
\usepackage{xspace}

\newtheorem{theorem}{Theorem}

\newtheorem{lemma}{Lemma}
\theoremstyle{definition}
\newtheorem{definition}{Definition}
\theoremstyle{remark}
\newtheorem{example}{Example}
\newtheorem{remark}{Remark}

\usepackage{hyperref}
\hypersetup{
  pdfauthor={Bahar Kalkan, Zijia Li, Hans-Peter Schröcker, Johannes Siegele},
  pdftitle={The Study Variety of Conformal Kinematics},
  pdfkeywords={Simple motion, Study variety, Study quadric, Null quadric, Four quaternion representation, Factorization},
  colorlinks=true,
  allcolors=blue,
}

\renewcommand{\H}{\mathbb{H}}
\renewcommand{\P}{\mathbb{P}}
\newcommand{\R}{\mathbb{R}}
\renewcommand{\DH}{\mathbb{DH}}
\newcommand{\NC}{\mathcal{C}}
\newcommand{\SQ}{\mathcal{Q}}
\newcommand{\SV}{\mathcal{S}}
\newcommand{\NQ}{\mathcal{N}}
\newcommand{\SC}{S}

\newcommand{\CGA}{\ensuremath{\operatorname{CGA}}\xspace}
\newcommand{\CGAp}{\ensuremath{\CGA_+}\xspace}
\newcommand{\OO}[1][3]{\ensuremath{\operatorname{O}(#1)}\xspace}
\newcommand{\SO}[1][3]{\ensuremath{\operatorname{SO}(#1)}\xspace}
\newcommand{\SE}[1][3]{\ensuremath{\operatorname{SE}(#1)}\xspace}
\newcommand{\Sim}[1][3]{\ensuremath{\operatorname{Sim}(#1)}\xspace}
\newcommand{\Em}[1][3]{\ensuremath{\operatorname{E}^{-}(#1)}\xspace}
\newcommand{\reverse}[1]{\widetilde{#1}}
\newcommand{\qi}{\mathbf{i}}
\newcommand{\qj}{\mathbf{j}}
\newcommand{\qk}{\mathbf{k}}
\newcommand{\eps}{\varepsilon}

\DeclareMathOperator{\Vect}{v}
\DeclareMathOperator{\Scal}{s}

\title[]{The Study Variety of Conformal Kinematics}

\date{\today}

\author{Bahar Kalkan}
\address{Van Yuzuncu Yil University, Department of Mathematics, Van, Turkey}
\email{baharkalkan@yyu.edu.tr}

\author{Zijia Li}
\address{KLMM, Academy of Mathematics and Systems Science,
  Chinese Academy of Sciences, Beijing, China}
\email{lizijia@amss.ac.cn}

\author{Hans-Peter Schröcker}
\address{Department of Basic Sciences in Engineering, University of Innsbruck,
  Innsbruck, Austria}
\email{hans-peter.schroecker@uibk.ac.at}

\author{Johannes Siegele}
\address{Department of Basic Sciences in Engineering, University of Innsbruck,
  Innsbruck, Austria}
\email{johannes.siegele@uibk.ac.at}

\keywords{
  Simple motion,
  Study variety,
  Study quadric,
  Null quadric,
  Four quaternion representation,
  Factorization}
\subjclass[2010]{
  15A66, 15A67, 20G20, 51B10, 51F15}

\begin{document}

\begin{abstract}
  We introduce the Study variety of conformal kinematics and investigate some of
  its properties. The Study variety is a projective variety of dimension ten and
  degree twelve in real projective space of dimension 15, and it generalizes the
  well-known Study quadric model of rigid body kinematics. Despite its high
  dimension, co-dimension, and degree it is amenable to concrete calculations
  via conformal geometric algebra (CGA) associated to three-dimensional
  Euclidean space. Calculations are facilitated by a four quaternion
  representation which extends the dual quaternion description of rigid body
  kinematics. In particular, we study straight lines on the Study variety. It
  turns out that they are related to a class of one-parametric conformal motions
  introduced by L.~Dorst in 2016. Similar to rigid body kinematics, straight
  lines (that is, Dorst's motions) are important for the decomposition of
  rational conformal motions into lower degree motions via the factorization of
  certain polynomials with coefficients in CGA.
\end{abstract}

\maketitle

\section{Introduction}
\label{sec:introduction}

Conformal Geometric Algebra (\CGA) is a commonly used algebra for describing
conformal displacements \cite{bayro-corrochano19,dorst07}. This paper is
dedicated to the study of conformal kinematics of three-space from an algebraic
and geometric viewpoint. Our aim is to describe a construction of a point model
of conformal kinematics, similar to the well-known Study quadric model of rigid
body kinematics, c.f. \cite[Chapter~11]{selig05}, to investigate some of its
properties, and to relate it to existing knowledge on \CGA. The Study quadric
model is based on the representation of the group \SE of rigid body
displacements by dual quaternions $\DH$. Rigid body displacements correspond to
points of a quadric $\SQ$ in $\P(\DH) = \P^7(\R)$, the \emph{Study quadric.} Its
equation is given by the condition that the dual quaternion norm is real. We
will define the \emph{Study variety $\SV$} of conformal kinematics as subvariety
in the projective space $\P^{15}(\R)$ over the even sub-algebra \CGAp of
Conformal Geometric Algebra of three-dimensional Euclidean geometry. Its ideal
encodes the well-known spinor conditions, which demand that the products of an
element $r \in \CGAp$ with its reverse from the left or from the right be real.
Our treatment is based on a \emph{four quaternion representation} for elements
of the 16-dimensional algebra \CGAp which allows transparent computations and
highlights the similarities to the ``two quaternion representation'' of \SE via
dual quaternions.

The group \SE of rigid body displacements is of high importance in robotics and
mechanism science and, because the dual quaternions $\DH$ form a sub-algebra of
$\CGAp$, is naturally embedded into conformal kinematics. While Clifford algebra
approaches to rigid body kinematics are, of course, well-known, c.f.
\cite{selig00} or \cite[Chapter~7]{bayro-corrochano19}, we observe that some
typical approaches to problems of robotics or mechanism science via $\DH$ have
not yet been generalized to conformal kinematics. Examples include the geometry
of the Study quadric and its relation to space kinematics in the sense of
\cite{selig05,rad18,lubbes18}, the study of constraint varieties
\cite{selig12,schroecker18,djintelbe21}, and the factorization theory of motion
polynomials \cite{hegedus13:_factorization2,li19:_motion_polynomials}. We feel
that it is worth extending or generalizing these concepts to \CGAp. This not
only teaches us about rigid body kinematics but also adds value to \CGAp, it
provides structure and new insight.

In particular, an important motivation when writing this text is to prepare the
ground for a factorization theory of ``spinor polynomials'', a conformal
generalization of motion polynomials. On the kinematics level, the
factorization corresponds to a decomposition of conformal motions with rational
trajectories into the product of motions of lower degrees. The basic building
blocks are linear polynomials. They correspond to straight lines on the Study
variety $\SV$ and, as we shall see in Section~\ref{sec:straight-lines}, they
parametrize a class of motions introduced by L.~Dorst in \cite{dorst16} as the
exponential of 2-blades (the wedge product of two vectors).

After recalling some well-known facts about the conformal geometric algebra
associated to 3D space in Section~\ref{sec:cga}, the subsequent
Section~\ref{sec:study-variety} features an investigation of the Study variety
$\SV$ based on a four quaternion representation of \CGAp. There we also
introduce the null quadric $\NQ$, a quadric in $\P^{15}$ which we consider
important for the geometric explanation of certain kinematic phenomena. One
instance of this can be found in Section~\ref{sec:straight-lines} where we
characterize straight lines through the identity on $\SV$, proof that they are
precisely the elementary motions investigated in \cite{dorst16}, and relate the
intersection points with $\NQ$ to Dorst's generation from two vectors. Combining
the concepts and methods developed in previous Sections with standard
factorization algorithms for Clifford algebras \cite{li18:_clifford_algebras}
already allows computing multiple decompositions of conformal motions given by
generic spinor polynomials into products of simple motions. We provide one
example in Section~\ref{sec:outlook}.

\section{Conformal Geometric Algebra and Dual Quaternions}
\label{sec:cga}

We follow the conventions of \cite[Chapter~8]{bayro-corrochano19} or
\cite{sommer2013geometric} for constructing the conformal model of
$n$-dimensional Euclidean space but we specialize to the case $n = 3$
immediately. We pick an orthonormal basis $\{e_1,e_2,e_3,e_+,e_-\}$ of
$\R^{4,1}$ with the properties
\begin{equation*}
  e_1^2 = e_2^2 = e_3^2 = e_+^2 = 1,\quad
  e_-^2 = -1.
\end{equation*}
The geometric product of two basis vectors is defined to be anticommutative:
\begin{equation*}
  e_ie_j = -e_je_i
  \quad\text{for pairwise different}\quad
  i,j \in \{1,2,3,+,-\}.
\end{equation*}
By linear extension, it generates the real algebra \CGA. It is customary, to
replace $e_+$, $e_-$ by $e_0$, $e_\infty$ via
\begin{equation*}
  e_o = \tfrac{1}{2}(e_- - e_+),\quad
  e_\infty = e_- + e_+
\end{equation*}
and to write $e_{ij}$ for $e_ie_j$, $e_{ijo\infty}$ for $e_ie_je_oe_\infty$ etc.
The \emph{reverse} $\reverse{e}_\ell$ of $e_\ell$ is obtained by inverting the
order of elements in $\ell$, that is, $\reverse{e}_{\ell_1 \cdots \ell_n} =
e_{\ell_n \cdots \ell_1}$. The \emph{grade} of the basis element $e_\ell$ with
$\ell \subset \{1,2,3,o,\infty\}$ is the cardinality of~$\ell$.

A general element of \CGA can be written as
\begin{equation}\label{eq:99}
  a = \sum_{\ell \subset \{1,2,3,o,\infty\}} a_\ell e_\ell,
  \quad\text{where $a_\ell \in \R$.}
\end{equation}
Its reverse is $\reverse{a} = \sum a_\ell\reverse{e}_\ell$. The even subalgebra
\CGAp of \CGA consists of all real linear combinations of even grade basis
elements. If all basis elements in \eqref{eq:99} are of grade one, $a$ is called
a \emph{vector}. The scalar product of two vectors is precisely the scalar
product in $\mathbb{R}^{4,1}$. Later, we will also need the wedge product of two
vectors, which is defined as $a \wedge b \coloneqq ab - a \cdot b$. Note that $a
\wedge a = 0$. A Euclidean point $(x_1,x_2,x_3)$ is represented in \CGA as the
vector
\begin{equation}\label{eq:point}
  x=e_0+x_1e_1+x_2e_2+x_3e_3 +\frac{x_1^2+x_2^2+x_3^2}{2} e_\infty,
\end{equation}
the point at infinity is represented by $e_\infty$. Vectors
$n_1e_1+n_2e_2+n_3e_3+de_\infty$ where the coefficient of $e_o$ is zero
represent planes with normal $n=(n_1,n_2,n_3)$ and oriented distance $d/||n||$
to the origin. All other vectors $a=\alpha(e_0+m_1e_1+m_2e_2+m_3e_3 +
1/2(m_1^2+m_2^2+m_3^2-\sigma) e_\infty)$ in \CGA represent spheres with midpoint
$(m_1,m_2,m_3)\in\R^3$ and radius
$r=\sqrt{\sigma}=\sqrt{a\reverse{a}/\alpha^2}$. If $\sigma$ is negative, the
sphere has an imaginary radius. It is natural to view spheres with zero radius as
Euclidean points with spheres of radius zero, that is, real multiples of vectors
in \eqref{eq:point}.

A rotor is often defined as an element of \CGAp satisfying $a\reverse{a} =
\reverse{a}{a} = 1$. Rotors form a double-cover of the group of direct conformal
displacements of conformally closed three-dimensional Euclidean space $\R^3 \cup
\{\infty\}$ which is isomorphic to $\SO[4,1]$
(c.f.~\cite[Chapter~8.4]{bayro-corrochano19}). We relax the rotor condition to
$a\reverse{a} = \reverse{a}a = \pm 1$, thus arriving at the algebra's spin
group, a double-cover of the group of conformal displacements which itself is
isomorphic to~$\OO[4,1]$. In order to get rid of the representation ambiguities,
we consider the spin group modulo the real multiplicative group $\R^\times$. We
call its elements \emph{homogeneous spinors.}

A homogeneous spinor is represented by an element $a \in \CGAp$ which satisfies
\begin{equation}
  \label{eq:1}
  a\reverse{a} = \reverse{a}a \in \R \setminus \{0\}
\end{equation}
This is our preferred interpretation as it gives rise to a model of conformal
kinematics as projective variety. As the representation of homogeneous spinors
is only unique up to non-zero scalar multiples, we can interpret homogeneous
spinors as points in $\CGAp$ modulo the real multiplicative group $\R^\times$.
This turns the \emph{vector space} \CGAp into the \emph{projective space}
$\P(\CGA_+)$. As the group of even-graded elements $\CGA_+$ is of dimension $16$
as a real vector space, we have $\P(\CGA_+) = \P^{15}(\R)$. The conformal group
$\OO[4,1]$ is embedded as a subvariety given by the equations arising from
$a\reverse{a} = \reverse{a}a \in \R$ minus the variety given by $a\reverse{a} =
\reverse{a}a = 0$. If $a \in \CGAp$ satisfies \eqref{eq:1}, we call
$a\reverse{a} = \reverse{a}a$ the \emph{norm of $a$.} As usual, equivalence
classes of non-zero scalar multiples are denoted by square brackets, i.e., $[a]
= [-a] = [2a]$ or $[1] = \R \setminus \{0\}$ etc. (Beware, this notation does
not distinguish between the equivalence class $[1]$ and a reference to a
bibliography item! Context will tell what is meant.)

Above construction is similar to the description of rigid body kinematics via
dual quaternions $\DH$ (c.f. \cite{husty12} or \cite[Section~9.3]{selig05}). The
real algebra of dual quaternions is generated from the basis elements
\begin{equation*}
  \qi,\
  \qj,\
  \qk,\
  \eps,\
  \qi\eps = \eps\qi,\
  \qj\eps = \eps\qj,\
  \qk\eps = \eps\qk
\end{equation*}
by the relations
\begin{equation*}
  \qi^2 = \qj^2 = \qk^2 = \qi\qj\qk = -1\quad\text{and}\quad \eps^2 = 0.
\end{equation*}
It is easy to verify that the dual quaternions are contained in $\CGA_+$ by the identifications
\begin{equation}
  \label{eq:2}
  \qi \mapsto -e_{23},\quad
  \qj \mapsto e_{13},\quad
  \qk \mapsto -e_{12},\quad\text{and}\quad
  \eps \mapsto e_{123\infty}.
\end{equation}
The norm of the dual quaternion $a = a_0 + a_1\qi + a_2\qj + a_3\qk + \eps(a_4 +
a_5\qi + a_6\qj + a_7\qk)$ is
\begin{equation*}
  a\reverse{a} =
  \reverse{a}a =
  a_0^2 + a_1^2 + a_2^2 + a_3^2 + 2\eps(a_0a_4 + a_1a_5 + a_2a_6 + a_3a_7)
\end{equation*}
and the real norm condition reduces to
\begin{equation}
  \label{eq:3}
  a_0a_4 + a_1a_5 + a_2a_6 + a_3a_7 = 0.
\end{equation}
The group of dual quaternions of non-zero norm modulo $\R^\times$ is isomorphic
to $\SE$. Its elements can be mapped bijectively onto the points of the
\emph{Study quadric} $\SQ \subset \P^7(\R)$ given by Equation~\eqref{eq:3} minus
the points of the \emph{null cone} $\NC \subset \P^7(\R)$ given by the vanishing
condition of the norm's real part
\begin{equation*}
  a_0^2 + a_1^2 + a_2^2 + a_3^2 = 0.
\end{equation*}

As explained in Section~\ref{sec:introduction}, this construction is relevant
for the geometry and algebra of $\SE$ and has applications in mechanism science
and robotics. In the next section, we extend it to the conformal group.

\section{Study Variety and Null Quadric}
\label{sec:study-variety}

The counterpart of the Study quadric $\SQ \subset \P^7$ in conformal kinematics
is the \emph{Study variety $\SV$,} a ten-dimensional projective variety in
$\P^{15}$ of degree twelve. It is complemented with the \emph{null quadric
  $\NQ$} which, in contrast to the null cone of the dual quaternion model of
\SE, is a regular quadric. This is quite important, as $SV \setminus \NQ$
consists of two disjoint components that can be thought of as direct and
indirect (orientation preserving and orientation reversing) conformal
displacements. For the sake of a transparent derivation of equations we suggest
a \emph{four quaternion representation} of \CGAp. It naturally extends the dual
(``two'') quaternion representation of~$\SE$.

\subsection{Four Quaternion Representation}

We distribute the 16 real coordinates of \CGAp into four groups of four
coordinates. The first group contains the coefficients of $1$, $-e_{23}$,
$e_{13}$, and $-e_{12}$, that is, it is a quaternion $r_0$ in the sense of
\eqref{eq:2}. The second group contains the coefficients of basis elements whose
index set contains $\infty$ but not $o$. With $\eps_1 \coloneqq e_{123\infty}$,
they are
\begin{equation*}
  \eps_1      = e_{123\infty},\quad
  \qi \eps_1  = -e_{23} \eps_1 = e_{1\infty},\quad
  \qj \eps_1  = e_{13} \eps_1 = e_{2\infty},\quad
  \qk \eps_1  = -e_{12} \eps_1 = e_{3\infty}.
\end{equation*}
We write the second group as $r_1\eps_1$ with a quaternion $r_1$. Note that
$\eps_1$ here equals the dual unit $\eps$ in \eqref{eq:2} so that $r_1$ is the
dual part of a dual quaternion. The third group contains the coefficients of
basis elements whose index set contains $o$ but not $\infty$. With $\eps_2
\coloneqq e_{123o}$ they are
\begin{equation*}
    \eps_2 = e_{123o},\quad
    \qi \eps_2 = -e_{23} \eps_2 = e_{1o},\quad
    \qj \eps_2 = e_{13} \eps_2 = e_{2o},\quad
    \qk \eps_2 = -e_{12} \eps_2 = e_{3o}.
\end{equation*}
We write this as $r_2\eps_2$ with a quaternion $r_2$. The fourth group contains
the remaining coefficients. Their respective index sets contain both, $\infty$
and $o$:
\begin{gather*}
    \eps_1\eps_2 = e_{\infty o},\quad
    \qi \eps_1\eps_2 = -e_{23} \eps_1\eps_2 = -e_{23\infty o},\\
    \qj \eps_1\eps_2 = e_{13} \eps_1\eps_2 = e_{13\infty o},\quad
    \qk \eps_1\eps_2 = -e_{12} \eps_1\eps_2 = -e_{12\infty o}.
\end{gather*}
We write this as $r_3\eps_1\eps_2$ with a quaternion~$r_3$.

Now, an element $q \in \CGAp$ has a unique representation as $q = r_0 +
r_1\eps_1 + r_2\eps_2 + r_3\eps_1\eps_2$ with quaternions $r_0$, $r_1$, $r_2$,
$r_3$. Note that $\reverse{\eps_1} = \eps_1$, $\reverse{\eps_2} = \eps_2$ but
$\reverse{\eps_1\eps_2} = -2 - \eps_1\eps_2$. In order to further simplify
reversion, we make the change of basis $\eps_1\eps_2 \mapsto \eps_3$ where
\begin{equation*}
  \eps_3 \coloneqq \eps_1\eps_2 + 1
\end{equation*}
so that $\reverse{\eps_3} = -\eps_3$. Note that $\eps_3 =
e_\infty\wedge e_o$, while $\eps_1\eps_2 = e_{\infty o}$. Again, there exist
unique quaternions $q_0 = r_0 - r_3$, $q_1 = r_1$, $q_2 = r_2$, $q_3 = r_3$ such
that $q = q_0 + q_1\eps_1 + q_2\eps_2 + q_3\eps_3$. This we call the \emph{four
  quaternion representation} of $q$.

Reversion in the four quaternion representation reads as
\begin{equation}
  \label{eq:4}
  \reverse{q} = \reverse{q_0} + \reverse{q_1}\eps_1 + \reverse{q_2}\eps_2 - \reverse{q_3}\eps_3.
\end{equation}
The product of $p = p_0 + p_1\eps_1 + p_2\eps_2 + p_3\eps_3$ and $s = s_0 +
s_1\eps_1 + s_2\eps_2 + s_3\eps_3$ is easily inferred from the multiplication
Table~\ref{tab:eps-multiplication} for $\eps_1$, $\eps_2$, and $\eps_3$ and the
fact that quaternions commute with $\eps_1$, $\eps_2$, and $\eps_3$. We have
\begin{equation}
  \label{eq:5}
  \begin{aligned}
    ps &= (p_0s_0 - p_1s_2 - p_2s_1 + p_3s_3) \\
    &+ (p_1(s_0 + s_3) + (p_0 - p_3)s_1)\eps_1 \\
    &+ (p_2(s_0 - s_3) + (p_0 + p_3)s_2)\eps_2 \\
    &+ (p_0s_3 + p_1s_2 - p_2s_1 + p_3s_0)\eps_3.
  \end{aligned}
\end{equation}

\begin{table}
  \centering
  \caption{Multiplication table for $\eps_1$, $\eps_2$, and $\eps_3$.}
  \label{tab:eps-multiplication}
  \begin{tabular}{c|ccc}
             & $\eps_1$      & $\eps_2$     & $\eps_3$  \\
    \hline
    $\eps_1$ & $0$           & $\eps_3 - 1$ & $\eps_1$  \\
    $\eps_2$ & $-\eps_3 - 1$ & $0$          & $-\eps_2$ \\
    $\eps_3$ & $-\eps_1$     & $\eps_2$     & $1$
  \end{tabular}
\end{table}

\subsection{Ideal of the Study Variety}
\label{sec:ideal}

The Study quadric of rigid body displacements is the projective variety in $\P^7
= \P(\DH)$ defined by the condition \eqref{eq:3} that its points have real norm.
When generalizing this to \CGAp we have to take into account that $q\reverse{q}
\neq \reverse{q}q$ in general so that the proper definition reads:

\begin{definition}
  The Study variety $\SV$ of conformal kinematics is the projective variety in
  $\P^{15} = \P(\CGAp)$ defined by the condition that both $q\reverse{q}$ and
  $\reverse{q}q$ are real.
\end{definition}

A projective point $[q] \in \SV$ describes a conformal displacement unless
$q\reverse{q} = \reverse{q}q = 0$. We will have a closer look at the algebraic
and geometric implications of this in Section~\ref{sec:null-quadric}. The
generating equations of the Study variety are obtained by plugging $p = q$, $s =
\reverse{q}$ and $p = \reverse{q}$, $s = q$ into \eqref{eq:5}:
\begin{equation}
  \label{eq:6}
  \begin{aligned}
    q\reverse{q} &= (q_0\reverse{q}_0 - q_1\reverse{q}_2 - q_2\reverse{q}_1 - q_3\reverse{q}_3) \\
    &+ (q_1(\reverse{q}_0 - \reverse{q}_3) + (q_0 - q_3)\reverse{q}_1)\eps_1 \\
    &+ (q_2(\reverse{q}_0 + \reverse{q}_3) + (q_0 + q_3)\reverse{q}_2)\eps_2 \\
    &- (q_0\reverse{q}_3 - q_1\reverse{q}_2 + q_2\reverse{q}_1 - q_3\reverse{q}_0)\eps_3,
  \end{aligned}
  \quad
  \begin{aligned}
    \reverse{q}q &= (\reverse{q}_0q_0 - \reverse{q}_1q_2 - \reverse{q}_2q_1 - \reverse{q}_3q_3) \\
    &+ (\reverse{q}_1(q_0 + q_3) + (\reverse{q}_0 + \reverse{q}_3)q_1)\eps_1 \\
    &+ (\reverse{q}_2(q_0 - q_3) + (\reverse{q}_0 - \reverse{q}_3)q_2)\eps_2 \\
    &+ (\reverse{q}_0q_3 + \reverse{q}_1q_2 - \reverse{q}_2q_1 - \reverse{q}_3q_0)\eps_3.
  \end{aligned}
\end{equation}
In order to write this more succinctly, we define
\begin{equation}
  \label{eq:study_condition}
 \SC(f,g) \coloneqq
 f\reverse{g} + g\reverse{f},
\end{equation}
and $\Vect(f) \coloneqq \frac{1}{2}(f-\reverse{f})$, the \emph{vector part}, for
quaternions $f$, $g$. With $f = f_0 + f_1\qi + f_2\qj + f_3\qk$ and $g = g_0 +
g_1\qi + g_2\qj + g_3\qk$ we have $\SC(f,g) = 2(f_0g_0 + f_1g_1 + f_2g_2 +
f_3g_3)$ so that $\SC(f,g) = 0$ is actually the Study condition for the dual
quaternion $f + \eps g$, compare with \eqref{eq:3}. In particular, $\SC(f,g)$ is
bilinear in the real coefficients of $f$ and $g$. Moreover, we have $\SC(f,g) =
\SC(g,f) = \SC(\reverse{f},\reverse{g}) = \reverse{f}g + \reverse{g}f$. Using
this notation, we can re-write the equations in \eqref{eq:6} as
\begin{equation}
  \label{eq:7}
  \begin{aligned}
    q\reverse{q} &= (q_0\reverse{q}_0 - \SC(q_1,q_2) - q_3\reverse{q}_3) \\
    &+ (\SC(q_0,q_1) - \SC(q_1,q_3))\eps_1 \\
    &+ (\SC(q_0,q_2) + \SC(q_2,q_3))\eps_2 \\
    &- 2(\Vect(q_0\reverse{q_3}) - \Vect(q_1\reverse{q_2}))\eps_3,
  \end{aligned}
  \quad
  \begin{aligned}
    \reverse{q}q &= (\reverse{q}_0q_0 - \SC(q_1,q_2) - \reverse{q}_3q_3) \\
    &+ (\SC(q_0,q_1) + \SC(q_1,q_3))\eps_1 \\
    &+ (\SC(q_0,q_2) - \SC(q_2,q_3))\eps_2 \\
    &+ 2(\Vect(\reverse{q_0}q_3) + \Vect(\reverse{q_1}q_2))\eps_3.
  \end{aligned}
\end{equation}
Note that the coefficients of $\eps_1$ and $\eps_2$ are real while the
coefficients of $\eps_3$ are vectorial quaternions. The constant coefficients
are the same and real as well. Thus, the ideal $\mathcal{I}$ of the Study
variety is generated by the vanishing conditions of the coefficients of
$\eps_1$, $\eps_2$, and $\eps_3$. We will denote an ideal generated by
polynomials $f_1$, $f_2,\ldots, f_n$ by $\langle f_1,f_2,\ldots,f_n \rangle$. By
a slight abuse of this notation, we can define
\begin{multline}
  \label{eq:8}
  \mathcal{I} \coloneqq \langle
  \SC(q_0,q_1), \SC(q_1,q_3), \SC(q_0,q_2), \SC(q_2,q_3), \Vect(q_0\reverse{q_3}) - \Vect(q_1\reverse{q_2}),
  \Vect(\reverse{q_0}q_3) + \Vect(\reverse{q_1}q_2)
  \rangle.
\end{multline}
Note that the last two ``generators'' are vectorial and actually represent six
real polynomials. We can get rid of them by using identities like
\begin{equation*}
  2\Vect(q_0\reverse{q_3}) = \SC(q_0, \qi q_3)\qi + \SC(q_0, \qj q_3)\qj + \SC(q_0, \qk q_3)\qk,
\end{equation*}
and thus arriving at
\begin{multline}
  \label{eq:9}
  \mathcal{I} \coloneqq \langle
  \SC(q_0,q_1), \SC(q_1,q_3), \SC(q_0,q_2), \SC(q_2,q_3),\\
  \SC(q_0,\qi q_3)-\SC(q_1,\qi q_2),
  \SC(q_0,\qj q_3)-\SC(q_1,\qj q_2),
  \SC(q_0,\qk q_3)-\SC(q_1,\qk q_2),\\
  \SC(q_0,q_3\qi)+\SC(q_1,q_2\qi),
  \SC(q_0,q_3\qj)+\SC(q_1,q_2\qj),
  \SC(q_0,q_3\qk)+\SC(q_1,q_2\qk)
  \rangle.
\end{multline}
We see that the ideal $\mathcal{I}$ of the Study variety is generated by ten
bilinear polynomials. Using computer algebra software it can readily be verified
that none of them can be removed without enlarging the Study variety. The
Hilbert polynomial of $\mathcal{I}$ is
\begin{multline*}
  H(x) =
   \tfrac{1}{302400}x^{10}
  +\tfrac{1}{7560}x^9
  +\tfrac{47}{20160}x^8
  +\tfrac{1}{42}x^7
  +\tfrac{2243}{14400}x^6\\
  +\tfrac{49}{72}x^5
  +\tfrac{121279}{60480}x^4
  +\tfrac{2963}{756}x^3
  +\tfrac{121883}{25200}x^2
  +\tfrac{709}{210}x
  +1.
\end{multline*}
We read off that the projective dimension of the Study variety equals $\dim\SV =
\deg H(x) = 10$. As expected, this equals $\dim\SO[4,1] = 10$. Moreover, the
degree of $\SV$ as projective variety equals $\deg\SV = \frac{1}{302400}\deg
H(x)! = 12$.

The ideal \eqref{eq:8} of $\SV$ may seem unwieldy at first sight. Nonetheless,
it has sufficient structure to allow for actual computations. We will encounter
examples of this later in this text. We summarize our findings in

\begin{theorem}
  \label{th:study-variety}
  The Study variety $\SV \subset \P^{15}$ of conformal kinematics is given by
  the ideal \eqref{eq:9} which is generated by ten bilinear polynomials. It is a
  projective variety of dimension ten and degree twelve.
\end{theorem}

\subsection{The Null Quadric}
\label{sec:null-quadric}

From \eqref{eq:7} we see that the real parts of $q\reverse{q}$ and
$\reverse{q}q$ are the same and vanish if and only if
\begin{equation}
  \label{eq:10}
  q_0\reverse{q_0} - \SC(q_1,q_2) - q_3\reverse{q_3} = 0.
\end{equation}
This equation defines a regular quadric $\NQ \subset \P^{15} = \P(\CGAp)$ which
we call the \emph{null quadric.} It extends the concept of the \emph{null cone
  $\NC$} in the dual quaternion model of rigid body kinematics. Indeed, when
plugging $q_2 = q_3 = 0$ into \eqref{eq:10} we obtain $q_0\reverse{q_0} = 0$,
the equation of a singular quadric of rank four in $\P^7 = \P(\DH)$.

Denote the ideal of the ``left'' Study condition $q\reverse{q} \in \R$ by
$\mathcal{L}$ and the ideal of the ``right'' Study condition $\reverse{q}q \in
\R$ by $\mathcal{R}$, c.f. \eqref{eq:7}. Using computational algebraic geometry
software, it can be shown that
\begin{equation*}
  \mathcal{L} = \mathcal{I} \cap \mathcal{L}',\quad
  \mathcal{R} = \mathcal{I} \cap \mathcal{R}'
\end{equation*}
where $\mathcal{L}'$ and $\mathcal{R}'$ are irreducible ideals of dimension ten
and degree 20 that both contain the polynomial defining the null quadric
$\mathcal{N}$. This tells us that the ``difference'' between the ``left'' and
the ``right'' Study variety is just in the null quadric $\NQ$. The non-null
points of $\SV$ are already determined by $\mathcal{L}$ or~$\mathcal{R}$.

As already mentioned, points of $\SV \setminus \NQ$ represent conformal
displacements. The algebraic closure of this set is $\SV$ while $\SV \cap \NQ$
can be thought of as its ``boundary''. The importance of boundaries in this
sense for questions of kinematics, in general, has been demonstrated in several
publications of the last decade, i.e.
\cite{djintelbe21,hegedus13:_bonds2,gallet15}. We will return to this in
Section~\ref{sec:straight-lines}.

\subsection{Kinematic Groups and Subvarieties}
\label{sec:groups-subvarieties}

Let us briefly explain how important groups of conformal and rigid body
kinematics can be described in terms of the four quaternion representation $q =
q_0 + q_1\eps_1 + q_2\eps_2 + q_3\eps_3$ and the Study variety ideal
$\mathcal{I}$ in the form~\eqref{eq:8}.

\begin{example}
  The \emph{special orthogonal group $\SO$} is encoded as $\mathcal{I} + \langle
  q_1, q_2, q_3 \rangle$. Because of bilinearity of the generators of
  $\mathcal{I}$ this simplifies to $\langle q_1, q_2, q_3 \rangle$. This just
  reconfirms the fact that $\SO$ is encoded by the quaternion group~$\H$.
\end{example}

\begin{example}
  \label{ex:SE3}
  The \emph{group of rigid body displacements $\SE$} is encoded as $\mathcal{I}
  + \langle q_2, q_3 \rangle$. This simplifies to the ideal $\langle
  \SC(q_0,q_1), q_2, q_3\rangle$ which is the well-known dual quaternion
  representation of $\SE$ \cite[Chapter~11]{selig05}. Composing $\SE$ with a
  fixed orientation reversing Euclidean displacement, for example the reflection
  in the origin, which is represented by $\eps_3$, yields the \emph{set $\Em$ of
    all orientation reversing Euclidean displacements.} From
  \begin{equation*}
    (q_0 + \eps_1 q_1) \eps_3 = q_1\eps_1 + q_0\eps_3
  \end{equation*}
  we see that $\Em$ is represented by the ideal $\langle \SC(q_1,q_3), q_0, q_2
  \rangle$.
\end{example}

\begin{example}
  The \emph{group $\Sim$ of direct similarities} is obtained by composing an
  element of $\SE$ with a uniform scaling. The scaling with factor
  $\sigma^{-1/2}$ is given by the homogeneous spinor $s = 2 +
  (1-\sigma)e_{\infty o} = s_0 + \eps_3 s_3$ where $s_0 = 1+\sigma$ and $s_3 =
  1-\sigma$. Writing the rigid body displacement as $r = p + \eps_1 d$ where
  $\SC(p,d) = 0$ we find
  \begin{equation}
    \label{eq:11}
    rs = (p + \eps_1 d)(s_0 + \eps_3 s_3) = ps_0 + d(s_0 + s_3)\eps_1 + ps_3\eps_3
    = ps_0 + 2d\eps_1 + ps_3\eps_3.
  \end{equation}
  The ideal of $\Sim$ is generated as $\langle \SC(q_0,q_1), q_2,
  \Vect(q_0\reverse{q_3}) \rangle$. The last condition encodes linear dependence
  of $q_0$ and $q_3$.
\end{example}

Note that neither $\Sim$ nor the group generated by $\SO$ and uniform scalings
are represented by projective subspaces of the Study variety $\SV$. This is
different for the subgroup generated by scalings and translations:

\begin{example}
  \label{ex:scalings-translations}
  Denote by $v = v_1\qi + v_2\qj + v_3\qk$ the translation vector. The
  composition of translation and scaling is obtained by plugging $p = 1$ and $d
  = -\frac{1}{2}v$ into \eqref{eq:11}, resulting in $s_0 - t\eps_1 + s_3\eps_3$.
  Thus, the ideal is
  \begin{equation*}
    \langle \Vect(q_0), \Scal(q_1), q_2, \Vect(q_3) \rangle
  \end{equation*}
  where $\Scal(\cdot)$ denotes the scalar part of a quaternion. This ideal is
  generated by eleven linear equations that indeed describe a projective
  subspace of dimension four.
\end{example}

\begin{example}
  \label{ex:5}
The composition of an inversion in the unit sphere, a translation, and
  another inversion in the unit sphere is called a \emph{transversion}
  \cite[Section~16.4]{dorst07}. More generally, we can replace the translation
  by an arbitrary element of $\SE$. This generates a subgroup isomorphic to
  $\SE$ whose elements have the four quaternion representation $q_0 + \eps_2
  q_2$. Transversions in the stricter sense of \cite[Section~16.4]{dorst07}
  appear for $q_0 \in \R$.
\end{example}

\section{Straight Lines on the Study Variety}
\label{sec:straight-lines}

The kinematic interpretation of straight lines on the Study quadric $\SQ$ is
well-known \cite[Section~11.2.1]{selig05}. Straight lines through the identity
displacement correspond, in general, to rotations around a fixed axis or, in
exceptional cases, to translations in a fixed direction. General straight lines
on $\SQ$ correspond to rotations or translations, composed with a fixed rigid
body displacement, either from the left or from the right.
For a given straight line trough the identity displacement $[1]$ we can chose an
arbitrary point $[q_0+\eps_1 q_1]\neq [1]$ with quaternions $q_0$, $q_1\in\H$ on
the line to obtain a parametric equation $t+q_0+\eps_1 q_1$. Here $t$ is a real
parameter, the point $[q_0+\eps_1 q_1]$ is obtained for $t=0$ while $[1]$ is
obtained in the limit $t\to\infty$.
The Study condition has to be fulfilled identically in $t$ so that $q_1$ is
necessarily vectorial. With $q = q_0 + \eps_1 q_1$ this is equivalent to $q +
\reverse{q} \in \R$. It is no loss of generality to assume that $q_0$ is
vectorial also, as otherwise we can re-parametrize via $t \mapsto t -
\Scal(q_0)$. This can be encoded as $q + \reverse{q} = 0$. In this section we
generalize these relations to \CGAp and its Study variety~$\SV$.

The cases of rotations and translations can be distinguished by the position of
the straight line with respect to the null cone $\NC$. Straight lines
corresponding to rotations intersect $\NC$ in a pair of conjugate complex points
while straight lines corresponding to translations intersect in the point
$[\eps_1q_1]$ with multiplicity two. As we shall see, the classification of
straight lines in the Study variety $\SV$ via the number of their real
intersection points with the null quadric $\NQ$ is natural in \CGAp as well.

We proceed with investigating straight lines through the identity displacement
and contained in $\SV$. In the light of above discussion, obvious examples come
from rotations around a fixed axis and translations in a fixed direction. They
are special cases of simple conformal motions described by L.~Dorst in
\cite{dorst16}. Slightly adapting Dorst's Equation~(2), we can define:

\begin{definition}
  \label{def:dorst-motion}
  A \emph{simple motion} is given by the exponential $\mathrm{e}^{uq}$, where $u$ is a
  real parameter and $q$ is a $2$-blade, that is $q = a \wedge b$, where $a$ and
  $b$ are vectors. This exponential is given by
  \begin{equation}
    \label{eq:12}
    \mathrm{e}^{uq} = \begin{cases}
\cos u + q \sin u \quad & \text{if $q\reverse{q} > 0$,} \\
      1 + q u \quad & \text{if $q\reverse{q} = 0$,} \\
      \cosh u + q \sinh u \quad & \text{if $q\reverse{q} < 0$}.
      \end{cases}
    \end{equation}
\end{definition}

We claim that elementary motions give examples of straight lines in~$\SV$:

Writing $a = a_o e_o + a_\infty e_\infty + a_1e_1 + a_2e_2 + a_3e_3$, $b = b_o
e_o + b_\infty e_\infty + b_1e_1 + b_2e_2 + b_3e_3$ and define $q_a \coloneqq
a_1\qi + a_2\qj + a_3\qk$, $q_b \coloneqq b_1\qi + b_2\qj + b_3\qk$. Replacing
$a$ or $b$ by a linear combination of $a$ and $b$ changes $q$ only up to an
irrelevant real factor. Thus, it is no loss of generality to assume $b_o = 0$.
We then have
\begin{equation}
  \label{eq:13}
  q = a \wedge b = -q_a \times q_b
  + (b_\infty q_a-a_\infty q_b) \eps_1
  - a_o q_b \eps_2
  - a_o b_\infty \eps_3.
\end{equation}
We see that $q$ has the four quaternion representation $q = q_0 + q_1\eps_1 +
q_2\eps_2 + q_3\eps_3$ with $\Scal(q_0) = \Scal(q_1) = \Scal(q_2) = \Vect(q_3) =
0$. It is easy to verify that the point $[q]$ satisfies the Study
conditions~\eqref{eq:9}. Any of the parametric equations \eqref{eq:12} is
equivalent to a parametric equation of type $t + q$ via a simple parameter
transformation (either $t = \cot u$, $t = u^{-1}$, or $t = \pm \coth u$). This
is, indeed, the parametric equation for a straight line through the identity
displacement. It is contained in $\SV$ by Dorst's construction (or by
Theorem~\ref{th:study-lines} below). Note that the identity displacement $[1]$
is only obtained in the limit for $\vert t \vert \to \infty$.

The main insight of this section is that all straight lines $t + q$ through the
identity and contained in $\SV$ are simple motions. We proceed by
deriving the conditions on $q$ that ensure that $[t+q] \in \SV$ for any $t \in
\R$ and then demonstrate that these conditions allow the decomposition $q = a
\wedge b$ with vectors $a$ and~$b$.

\begin{theorem}
  \label{th:study-lines}
  The straight line $t + q$ with $q = q_0 + q_1\eps_1 + q_2\eps_2 + q_3\eps_3$
  is contained in the Study variety $\SV$ if and only if
  $[q] \in \SV$
  and $q + \reverse{q} \in \R$. If $\Scal(q_0) = 0$, then this is equivalent to
  $q + \reverse{q} = 0$.
\end{theorem}

\begin{proof}
  Recall that $[q] \in \SV$ is equivalent to $q\reverse{q} = \reverse{q}q \in
  \R$. Now, $(t+q)(t+\reverse{q})$ and $(t+\reverse{q})(t+q)$ need to be real
  for any $t \in \R$. From
  \begin{equation*}
    (t+q)(t+\reverse{q}) =
    (t+\reverse{q})(t+q) =
    t^2 + (q + \reverse{q})t + \reverse{q}q
  \end{equation*}
  we infer that this is the case if and only if $q + \reverse{q}$ and
  $\reverse{q}q$ is real. The latter condition implies $[q] \in \SV$ and the
  theorem's first claim follows. The second can be read off from the explicit
  representation
  \begin{equation}
    \label{eq:14}
    \frac{1}{2}(q + \reverse{q}) = \Scal(q_0) + \Scal(q_1)\eps_1 + \Scal(q_2)\eps_2 + \Vect(q_3)\eps_3.
    \qedhere
  \end{equation}
\end{proof}

\begin{lemma}
  \label{lem:wedge}
  If $q + \reverse{q} = 0$ there exist vectors $a$, $b$ such that $q = a \wedge
  b$.
\end{lemma}

\begin{proof}
  The case $q = 0$ is trivial and not of interest to us. We will exclude it
  for the remainder of this proof.

  With $q \neq 0$ given as $q = q_0 + q_1\eps_1 + q_2\eps_2 + q_3\eps_3$,
  \eqref{eq:13} implies that we have to solve
  \begin{equation}
    \label{eq:15}
    q_0 = -q_a \times q_b,\quad
    q_1 = b_\infty q_a-a_\infty q_b,\quad
    q_2 = - a_o q_b,\quad
    q_3 = - a_o b_\infty
  \end{equation}
  for the scalars $a_0$, $a_\infty$, $b_\infty$ and the vectors $q_a$ and $q_b$.
  From \eqref{eq:14} we see that $q_0$, $q_1$, and $q_2$ are vectorial and $q_3$
  is scalar. This is an obvious necessary condition for \eqref{eq:15} to have a
  solution. We distinguish the two cases $q_3 = 0$ and $q_3 \neq 0$.

  Case 1: $q_3 = 0$: If $q_2 = 0$ (Case~1.1; rigid body displacements), then
  either $a_o = 0$ or $b_\infty = q_b = 0$. The latter implies $b = 0$ and is
  not possible because we assumed $q \neq 0$. Thus, $a_o = 0$. Now, we are left
  with the two equations
    \begin{equation*}
      q_0 = -q_a \times q_b,\quad
      q_1 = b_\infty q_a-a_\infty q_b.
    \end{equation*}
    We view them as vector equations in $\R^3$. If $q_0 = 0$, we can pick $q_a$
    and $q_b$ as scalar multiples of $q_1$ to satisfy the first equation and
    then solve the second equation for $a_\infty$ and $b_\infty$. If $q_0 \neq
    0$, the first equation admits infinitely many solutions for $q_a$ and $q_b$,
    all of them in the orthogonal complement $q_0^\perp$ of $q_0$ and linearly
    independent. By the Study conditions on $q$ we have $\SC(q_0,q_1) = 0$ so
    that $q_1 \in q_0^\perp$ as well. Therefore, given solutions $q_a$ and $q_b$
    of the first equation, the second equation can be solved for $a_\infty$
    and~$b_\infty$.

    If $q_2 \neq 0$ (Case 1.2), then $a_o \neq 0$ so that $b_\infty = 0$. We set
    $a_o = -1$ and $q_b = q_2$ to satisfy the third equation of \eqref{eq:15}.
    The remaining equations are
    \begin{equation*}
      q_0 = -q_a \times q_2,\quad
      q_1 = -a_\infty q_2.
    \end{equation*}
    By \eqref{eq:8}, points of the Study variety satisfy
    $\Vect(q_1\reverse{q_2})-\Vect(q_0\reverse{q_3}) = 0$ which simplifies to
    $q_1q_2 \in \R$, or equivalently $q_1 \times q_2 = 0$, in our case. Thus,
    $q_1$ and $q_2$ are linearly dependent as vectors in $\R^3$. Since $q_2 \neq
    0$, the second equation can be solved uniquely for $a_\infty$ while the
    first equation admits infinitely many solutions for~$q_a$.

    Case~2: $q_3 \neq 0$. We set $a_o = -1$, $b_\infty = q_3$, and $q_b = q_2$
    to satisfy the third and fourth equation of \eqref{eq:15}. The two remaining
    equations read
    \begin{equation*}
      q_0 = -q_a \times q_2, \quad
      q_1 = q_3q_a - a_\infty q_2.
    \end{equation*}
    Again, we view them as vector equations in $\R^3$. Solving the second
    equation for $q_a$ yields $q_a = (q_1 + a_\infty q_2)/q_3$. Plugging this
    into the first equation leads to $q_0q_3 = -q_1 \times q_2$. This holds true
    as it is only an alternative way of writing the condition
    $\Vect(q_0\reverse{q_3}) - \Vect(q_1,\reverse{q_2}) = 0$ from \eqref{eq:8}.

    After having discussed all possible cases and sub-cases, the proof is
    complete.
\end{proof}

\begin{theorem}
  \label{th:dorst-motions}
  Elementary Dorst motions correspond to straight lines on $\SV$ through the
  identity displacement $[1]$ and vice versa.
\end{theorem}

\begin{proof}
  We have already argued that a simple motion can be parameterized as
  $t + q$ where $[q] \in \SV$ and $q + \reverse{q} = 0$. This is, indeed, a
  straight line through $[1]$. Conversely, given a straight line on $\SV$ and
  through $[1]$ it can always be parametrized as $t + q$ with $q + \reverse{q}
  \in \R$ by Theorem~\ref{th:study-lines}. The admissible re-parametrization $t
  \mapsto t - \Scal(q_0)$ then implies $q + \reverse{q} = 0$ whence $q = a
  \wedge b$ with vectors $a$, $b$ by Lemma~\ref{lem:wedge}. But then $t + q$ is
  a simple motion by Definition~\ref{def:dorst-motion}.
\end{proof}

The decomposition $q = a \wedge b$ of Lemma~\ref{lem:wedge} is not unique.
Replacing vectors $a$ with the linear combination $\alpha_1 a+\beta_1 b$ and $b$
by $\alpha_2 a+\beta 2 b$, for $\alpha_1$, $\alpha_2$, $\beta_1$, $\beta 2\in\R$
with $\alpha_1\beta_2-\alpha_2\beta_1=1$ will result in the same $q$. By taking
independent linear combinations of $a$ and $b$, we will obtain a scalar multiple
of $q$, which represents the same point on the Study variety. General vectors
$a$ and $b$ represent spheres in \CGA. The possibility to pick special spheres,
i.e., planes or points, in the decomposition $q = a \wedge b$ of $q$ gives rise
to six types of simple motions \cite{dorst16}: conformal rotations, conformal
scalings, transversions, rotations, translations and uniform scalings.

In our context, it is maybe more natural to primarily distinguish three types of
straight lines $t + q$ through the identity displacement and on $\SV$ according
to the number $n$ of real intersection points with the null quadric $\NQ$. They
correspond to the Dorst cases conformal rotation ($n = 0$), transversion ($n =
1$) in a more general sense than in Example~\ref{ex:5}, and conformal scaling
($n = 2$). Sub-cases with special Euclidean relevance are Euclidean rotation ($n
= 0$), translation ($n = 1$), and uniform scaling ($n = 2$). They are easy to
distinguish in the four quaternion representation of $q = q_0 + q_1\eps_1 +
q_2\eps_2 + q_3\eps_3$, c.f. Examples~\ref{ex:SE3} and
\ref{ex:scalings-translations}:
\begin{itemize}
\item Euclidean rotations among conformal rotations and translations among
  transversions are characterized by $q_2 = q_3 = 0$, that is, they are rigid
  body displacements.
\item Uniform scalings are characterized among conformal scalings by $\Vect(q_0)
  = \Scal(q_1) = q_2 = \Vect(q_3) = 0$. The condition $\Scal(q_1) = 0$ is
  already implied by the fact that $t + q$ is contained in the Study
  variety~$\SV$.
\end{itemize}

The six simple motions can also be distinguished by the number $f \in
\{0,1,2\}$ of their real fixpoints and, if $f \ge 1$, whether $e_\infty$ is
among them. The relation of fixpoints to the intersection points of $t + q$ with
$\NQ$ is investigated next.

\subsection*{Conformal Rotations and Conformal Scalings}
\label{sec:conformal-rotations-scalings}

Given a straight line $t + q$ on $\SV$ and through the point $[1]$ we already
argued (Lemma~\ref{lem:wedge}) that $q$ can be chosen as $q = a \wedge b$ with
vectors $a$, $b$. In the two generic cases of Dorst's classification (conformal
rotation and conformal scaling), $a$ and $b$ are two points --- conjugate
complex in case of conformal rotations and real in case of conformal scalings.
We compute the parameter values $t_1$, $t_2$ of the intersection points $[n_1]$,
$[n_2]$ with the null quadric $\NQ$. From
\begin{equation*}
  (t + q)(t + \reverse{q}) = (t + q)(t - q) = t^2 - q^2 = t^2 - (a \cdot b)^2
\end{equation*}
we infer $t_1 = a \cdot b$ and $t_2 = -a \cdot b$ and hence $[n_1] = [ab]$,
$[n_2]=[-ba]=[ba]$.

Consider now, at least formally, the $n_1$-image of a point $x$. We compute
\begin{equation*}
  y = n_1x\reverse{n_1} = 4 (a \cdot b)(b \cdot x) a
\end{equation*}
so that the $n_1$-image of a point $x$ equals $a$ unless $b \cdot x = 0$, i.e.,
$x$ is ``perpendicular'' to $b$. In this sense, the ``null displacement'' $n_1$
has the generic image point $a$. The generic image point for $n_2$ is~$b$.

\subsection*{Uniform Scalings}
\label{sec:uniform-scalings}

A similar computation is possible for $q = a \wedge b$ where $a$ represents the
point with coordinates $(a_1,a_2,a_3)$ and $b = e_\infty$. According to
\cite{dorst16}, this is a uniform scaling with center $a$ (an ``isotropic''
scaling, as it is called there). We have
\begin{equation*}
  q = a_1e_{1\infty} + a_2e_{2\infty} + a_3e_{3\infty} - e_{\infty o} - 1
  \quad\text{and}\quad
  (t + q)(t + \reverse{q}) = t^2 - 1
\end{equation*}
so that the intersection points with $\NQ$ are
\begin{equation*}
  \begin{aligned}
    [n_1] &= [-1+q] = [a_1e_{1\infty} + a_2e_{2\infty} + a_3e_{3\infty} - e_{\infty o} - 2] = [ae_\infty],\\
    [n_2] &= [1+q] = [a_1e_{1\infty} + a_2e_{2\infty} + a_3e_{3\infty} - e_{\infty o}] = [-\reverse{n_1}] = [\reverse{n_1}].
\end{aligned}
\end{equation*}
Acting with $n_1$ and $n_2$ on a generic point $x$ gives
\begin{equation*}
  n_1x\reverse{n_1} = 4a,
  \quad
  n_2x\reverse{n_2} = -4(a \cdot x) e_\infty,
\end{equation*}
respectively. We see that the action of $n_1$ on $x$ always gives point $a$. The
only exception is $x = e_\infty$ where the action is undefined due to
$n_1e_\infty\reverse{n_1} = 0$. The action of $n_2$ on $x$ gives $e_\infty$
unless $a \cdot x = 0$.

A more elementary interpretation for $n_1$ and $n_2$ is scalings with center $a$
and respective scaling factors $0$ and $\infty$. It is very intuitive to say
that the zero-scaling $n_1$ fixes $e_\infty$ and maps everything else to $a$.
Likewise, the $\infty$-scaling $n_2$ fixes $a$ and maps everything else
to~$e_\infty$.

\subsection*{Transversions and Translations}
\label{sec:transversions-translations}

A transversion is given by $q = a \wedge b$ where $a$ represents a point
and $b$ is a plane perpendicular to $a$. If $a \neq e_\infty$ we may write $a =
(a_1,a_2,a_3)$ whence
\begin{equation*}
  b = b_1 e_1 + b_2 e_2 + b_3 e_3 + (a_1 b_1 + a_2 b_2 + a_3 b_3) e_\infty.
\end{equation*}
Because of $a \cdot b = 0$ we have $q = a \wedge b = ab$. Because $a$
and $b$ are both vectors, we further have $\reverse{q}= ba$ so that
$q\reverse{q} = \reverse{q}q = abba = a^2b^2 = 0$, because $a^2 = 0$. By
\eqref{eq:13} $q + \reverse{q} = 0$ so that $(t+q)(t+\reverse{q}) = t^2$ whence
$q$ is, indeed, the only intersection point of the straight line $t+q$ with
$\NQ$. Acting on a generic point $x$ yields
\begin{equation*}
  qx\reverse{q} = -2 b^2 (a \cdot x) a.
\end{equation*}
We infer that $a$ is the image of point $x$ unless $x$ is perpendicular to~$a$.
A similar computation is also possible for $a = e_\infty$. It gives a
translation.

Summarizing our findings in
Sections~\ref{sec:conformal-rotations-scalings}--\ref{sec:transversions-translations},
we can say that a simple motion, represented by a straight line $t +
q$ on $\SV$ and through the point $[1]$ has up to two exceptional real points
that are obtained as images of a generic point under the displacements obtained
as real intersection points of $t + q$ with the null quadric~$\NQ$. They are
fixed by the motion but the set of fixed points is generically larger.

\begin{remark}
  Euclidean rotations are somewhat exceptional in this context. Here $q = a
  \wedge b$ with \emph{planes} $a$ and $b$. In the projective closure of
  Euclidean three-space, one might be tempted to consider planes with $a^2 = b^2
  = 0$ as complex points at infinity. This viewpoint is uncommon in the
  conformal closure of Euclidean three-space.
\end{remark}

\section{Outlook}
\label{sec:outlook}

An important motivation for our study is extension of the factorization theory
of motion polynomials
\cite{hegedus13:_factorization2,li18:_clifford_algebras,li19:_motion_polynomials}
to ``spinor polynomials'' of conformal geometric algebra. These are defined as
univariate polynomials $C$ with coefficients in $\CGAp$ such that $C\reverse{C}
= \reverse{C}C$ is a real polynomial different from $0$. They describe conformal
motions where all point trajectories are rational curves. Factorization of $C$
into linear factors corresponds to the decomposition of the conformal motion
into concatenated simple motions. Results of \cite{li18:_clifford_algebras}
essentially imply that such factorizations do exist for ``generic'' spinor
polynomials. Typically, they are not unique. We present one example:

\begin{example}
  \label{ex:factorization}
  The polynomial $C = (t + h_1)(t + h_2)(t + h_3)$ with
  \begin{equation*}
    h_1 = - e_{1\infty} + 2e_{1o},\quad
    h_2 = - e_{2\infty} - 2e_{2o},\quad
    h_3 = 1 - \tfrac{1}{2} e_{3\infty} + e_{3o} + e_{\infty o}
  \end{equation*}
  satisfies
  $C\reverse{C} = M_1 M_2 M_3$ where
  \begin{equation*}
    \begin{aligned}
      M_1 &\coloneqq (t + h_1)(t + \reverse{h}_1) = t^2 + 4,\\
      M_2 &\coloneqq (t + h_2)(t + \reverse{h}_2) = t^2 - 4,\\
      M_3 &\coloneqq (t + h_3)(t + \reverse{h}_3) = t^2.
    \end{aligned}
  \end{equation*}
  It is thus a spinor polynomial. We see that it is the composition of a
  transversion, a conformal scaling, and a conformal rotation. They correspond,
  in that order, to $h_3$, $h_2$, and $h_1$. Using Algorithm~2 of
  \cite{li18:_clifford_algebras} we can compute further factorizations $C =
  (t+k_1)(t+k_2)(t+k_3) = (t+\ell_1)(t+\ell_2)(t+\ell_3)$ where
  \begin{equation*}
    \begin{aligned}
      k_1 &= -e_{2\infty} - 2 e_{2 o},\\
      k_2 &= -1 + \tfrac{1}{2} e_{3\infty} - e_{3 o} - e_{\infty o},\\
      k_3 &= \phantom{-}2 - e_{1\infty} + 2 e_{1 o} - e_{3\infty} + 2 e_{3 o} + 2 e_{\infty o},
    \end{aligned}
  \end{equation*}
  and
  \begin{equation*}
    \begin{aligned}
      \ell_1 &= \phantom{-}1 + e_{2 3} - \tfrac{1}{2} e_{2 \infty} - e_{2 o} + \tfrac{1}{2} e_{3 \infty} + e_{3 o},\\
      \ell_2 &= -2 - e_{2 3} - \tfrac{1}{2} e_{2 \infty} - e_{2 o} - 2 e_{3 o} - e_{\infty o},\\
      \ell_3 &= \phantom{-}2 - e_{1 \infty} + 2 e_{1 o} - e_{3 \infty} + 2 e_{3 o} + 2 e_{\infty o}.
  \end{aligned}
  \end{equation*}
  Because of $(t+k_1)(t+\reverse{k}_1) = M_2$, $(t+k_2)(t+\reverse{k}_2) = M_3$,
  and $(t+k_3)(t+\reverse{k}_3) = M_1$ the factors $k_3$, $k_2$, and $k_1$
  correspond, in that order, to a conformal rotation, a transversion, and a
  conformal scaling. Finally, we have $(t+\ell_1)(t+\reverse{\ell}_1) = t(t+2)$,
  $(t+\ell_2)(t+\reverse{\ell}_2) = t(t-2)$, $(t+\ell_3)(t+\reverse{\ell}_3) =
  M_1$, the third factorization correspond to the decomposition into a conformal
  rotation and two conformal scalings.\hfill $\diamond$
\end{example}

Algorithm~2 of \cite{li18:_clifford_algebras} produces indeed, the maximal
number of twelve factorizations of the polynomial $C$ of
Example~\ref{ex:factorization}. Each of these factorizations corresponds to one
of the twelve ways of writing $C\reverse{C}$ as product of real quadratic 
polynomials. The algorithm relies on the fact that certain algebra elements,
obtained as leading coefficients of linear remainders of polynomial
division, are invertible. This is generic behavior but non-generic examples are
assumed to exist. Factorizability results in this sense will be the topic of a
future publication.

\section{Conclusion}
\label{sec:conclusion}

We have introduced Study variety $\SV$ and null quadric $\NQ$ of conformal
kinematics and investigated some of their properties. The structure of their
ideals becomes clearer in our four quaternion representation of \CGAp. We also
use the four quaternion representation for investigating straight lines on
$\SV$. They turn out to be related to elementary motions suggested by L.~Dorst
in \cite{dorst16}.

We view our results as generalizations of the representation of $\SE$, the
groups of rigid body displacements, where dual quaternions give rise to a
kinematic from $\SE$ to the points of the Study quadric $\SQ$ minus the null
cone $\NC$. Straight lines on $\SQ$ are known to correspond to rotations around
a fixed axis or translations in a fixed direction. They appear naturally in the
factorization theory of motion polynomials \cite{hegedus13:_factorization2} and
we have pointed at similar relations between ``spinor polynomials'' and straight
lines on $\SV$, that is, simple motions.

\section*{Acknowledgement} 

Bahar Kalkan was supported by the BIDEB 2211-E scholarship programme of The
Scientific and Technological Research Council of Turkey. Johannes Siegele was
supported by Austrian Science Fund (FWF) P~33397-N (Rotor Polynomials: Algebra
and Geometry of Conformal Motions).

\end{document}